\documentclass[a4paper, notitlepage, oneside, reqno]{amsart}
\usepackage{amsmath,amsfonts,amssymb}
\usepackage{amsthm} 
\usepackage{graphicx} 
\usepackage{subfig}
\usepackage{listings}
\usepackage{geometry}
\usepackage{hyperref}
\usepackage{wrapfig}
\usepackage{url}
\usepackage{tikz}
\usetikzlibrary{shapes, positioning, arrows}
\usepackage{eso-pic}
\usepackage{mathrsfs}
\usepackage{xcolor}
\usepackage{mathtools}

\newcommand{\bigO}{\mathcal O}

\newcommand{\dt}{\, \mathrm{d}t}
\newcommand{\dx}{\, \mathrm{d}x}
\newcommand{\dy}{\, \mathrm{d}y}

\newcommand{\dxi}{\, \mathrm{d}\xi}

\newcommand{\R}{\mathbb{R}}

\newcommand{\N}{\mathbb{N}}

\newcommand{\T}{\mathbb{T}}

\DeclareMathOperator{\F}{{\mathscr F}}

\DeclareMathOperator{\supp}{supp}

\newtheorem{theorem}{Theorem}[section]
\newtheorem{lemma}[theorem]{Lemma}
\newtheorem{corollary}[theorem]{Corollary}
\newtheorem{proposition}[theorem]{Proposition}

\theoremstyle{definition}

\newtheorem{remark}[theorem]{Remark}
\numberwithin{equation}{section}

\title{Non-uniform dependence on initial data for equations of Whitham type}
\author{Mathias Nikolai Arnesen}
\date{\today}

\thanks{The author gratefully acknowledges the support of the project Nonlinear water waves (Grant No. 231668) from the Research Council of Norway.}

\begin{document}

\begin{abstract}
We consider the Cauchy problem
\begin{align*}
\partial_t u+u\partial_x u+L(\partial_x u) &=0, \\
u(0,x)=u_0(x)
\end{align*}
on the torus and on the real line for a class of Fourier multiplier operators $L$, and prove that the solution map $u_0\mapsto u(t)$ is not uniformly continuous in $H^s(\T)$ or $H^s(\R)$ for $s>\frac{3}{2}$. Under certain assumptions, the result also hold for $s>0$.  The class of equations considered includes in particular the Whitham equation and fractional Korteweg-de Vries equations and we show that, in general, the flow map cannot be uniformly continuous if the dispersion of $L$ is weaker than that of the KdV operator. The result is proved by constructing two sequences of solutions converging to the same limit at the initial time, while the distance at a later time is bounded below by a positive constant.
\end{abstract}

\maketitle

\section{Introduction}
We consider the Cauchy problem
\begin{subequations}
	\begin{gather}
		\partial_t u+u\partial_x u+L(\partial_x u)=0, \label{eq: whitham} \\
		u(x,0)=u_0(x), \,\, x\in \R \,\, \text{or} \,\, x\in \T, \,\, t\in \R. \label{eq: whitham id}
	\end{gather}
\end{subequations}
on the torus $\T$ and on the real line $\R$. The operator $L$ is a Fourier multiplier operator with symbol $m(\xi)$, meaning that
\begin{equation}
\widehat{L f}(\xi)=m(\xi)\widehat{f}(\xi).
\end{equation}
A concrete example is the Whitham equation where $m(\xi)=\sqrt{\frac{\tanh(\xi)}{\xi}}$. The Whitham equation was introduced by Whitham in 1967 as a better alternative to the Korteweg--de Vries (KdV) equation for modelling shallow water waves \cite{Whitham1967vma}, and features the exact linear dispersion relation for travelling gravity water waves (see \cite{Lannes2013twp} for a rigorous justification of \eqref{eq: whitham} as a model for shallow water waves and \cite{Moldabayev2015twe} for a derivation of it from the Euler equations via exponential scaling). 

The recent papers \cite{EEP} and \cite{Ehrnstrom2016olw} concern local well-posedness for the Whitham equation and related nonlinear and nonlocal dispersive equations with nonlinearities of low regularity. These results are, when comparable, in line with the earlier investigations \cite{Abdelouhab1989nmf} \cite{Saut1976rot}. For problems with homogeneous symbols and smooth nonlinearities, local well-posedness has been lowered to $s\leq \frac{3}{2}$ in \cite{Linares2014dpo} using dispersive properties, with the lower bound for $s$ depending on the strength of the dispersion. The paper at hand concerns further regularity of the flow map, or rather the lack thereof. We prove that in the periodic case, the flow map is not uniformly continuous on any bounded set of $H^s(\T)$ for $s>\frac{3}{2}$ for any symbol $m$ that is even and locally bounded, and on the real line the flow map is not uniformly continuous on any bounded set of $H^s(\R)$ for $s>\frac{3}{2}$ if $m$ does not grow "too" quickly. The results are also extended to $0<s\leq \frac{3}{2}$ under certain conditions. The paper is motivated by a series of similar results for other model equations (e.g., for the Camassa--Holm (CH) equation \cite{Himonas2009nud} and the Benjamin--Ono (BO) equation \cite{Koch2005nwi}), and indeed for the Euler equations themselves \cite{Himonas2010nud}, as well as recent investigations into non-local dispersive equations of Whitham type with very general, and in particular also inhomogeneous, symbols $m$ (\cite{EGW},\cite{Arnesen2016eos}, \cite{Ehrnstrom2016olw}) and recent work connected to well-/ill-posedness for the Whitham equation specifically. In particular, we mention the recent positive verifications of two conjectures of Whitham, namely that for certain initial data the solution exhibits wave-breaking in finite time \cite{Hur2015bit} and the existence of a highest cusped wave \cite{Ehrnstrom2015owc}.

Our results are contained in Theorems \ref{thm: periodic} and \ref{thm: line} below, for the periodic case and the real line case, respectively.
\begin{theorem}[Non-uniform continuity on $\T$]
\label{thm: periodic}
Assume that $m\in L_{loc}^\infty (\R)$ is even, there exists $N>0$ such that $m(\xi)$ is continuous for $|\xi|>N$ and that 
\begin{equation*}
|m(\xi)|\lesssim |\xi|^p
\end{equation*}
for some $p>0$ when $|\xi|\gg 1$. Then:
\begin{itemize}
\item[(i)] If $s>\frac{3}{2}$, the flow map $u_0\mapsto u(t)$ for the Cauchy problem \eqref{eq: whitham}-\eqref{eq: whitham id} on the torus is not uniformly continuous from any bounded set in $H^s(\T)$ to $C([0,T);H^s(\T))$.
\item[(ii)] Let $0<s\leq \frac{3}{2}$. When the flow map exists on $H^s(\T)$, (i) remains true.
\end{itemize}
\end{theorem}

\begin{theorem}[Non-uniform continuity on $\R$]
\label{thm: line}
Assume that $m$ is even, $m\in L_{loc}^\infty (\R)$ and there exists $N>0$, $0\leq \gamma<2$ and a constant $C>0$ such that 
\begin{equation}
\label{eq: cond. m}
|m(\xi+y)-m(\xi)|\leq C |y| |\xi|^{\gamma-1}
\end{equation}
for all $|\xi|>N$ and $|y|$ sufficiently small. In particular, this means that $m(\xi)$ is continuous for $|\xi|>N$ and that $|m(\xi)|\lesssim |\xi|^\gamma$ for large $|\xi|$. Then:
\begin{itemize}
\item[(i)] If $s>\frac{3}{2}$, the flow map $u_0\mapsto u(t)$ for the Cauchy problem
\eqref{eq: whitham}-\eqref{eq: whitham id} on the line is not uniformly continuous from any bounded set in $H^s(\R)$ to $C([0,T);H^s(\R))$.
\item[(ii)] Let $\frac{1}{2}<r<2$ and $0<s< \frac{r}{2}$. If the Cauchy problem \eqref{eq: whitham}-\eqref{eq: whitham id} is locally well-posed in $H^s(\R)$ in the sense of Theorem \ref{thm: local well-posedness} and $m$ satisfies the lower bound
\begin{equation*}
|m(\xi)|\gtrsim |\xi|^r
\end{equation*}
for $|\xi|\gg 1$ in addition to \eqref{eq: cond. m}, then (i) is true also for $0<s< \frac{r}{2}$.
\end{itemize}
\end{theorem}

\begin{remark}
Local well-posedness of \eqref{eq: whitham} is in general known only for $s>\frac{3}{2}$ (cf. Theorem \ref{thm: local well-posedness} below), hence we have to assume the existence of the flow map in part (ii) of Theorem \ref{thm: periodic} and \ref{thm: line}. For some specific choices of $m$ well-posedness results for $s\leq \frac{3}{2}$ are known (see for instance \cite{Linares2014dpo}).
\end{remark}
\begin{remark}
The additional condition $|m(\xi)|\gtrsim |\xi|^r$ for $|\xi|\gg 1$ and the bounds on $s$ in Theorem \ref{thm: line} (ii) come from using conservation laws for \eqref{eq: whitham} to bound the $H^\sigma(\R)$ norm of the solution in terms of the norm of the initial data for $\sigma>s$ (cf. the end of Section \ref{sec: line}). These conditions can be improved upon in cases where more conservation laws are known, as done in \cite{Koch2005nwi} for the BO equation in order to cover all $s>0$.
\end{remark}

The assumptions of Theorems \ref{thm: periodic} and \ref{thm: line} cover, for example, the Whitham equation, and the fractional Korteweg--de Vries (fKdV) equation where $m(\xi)=|\xi|^\alpha$ for any $\alpha\geq 0$ in the periodic case and $0\leq \alpha<2$ on $\R$. One would expect the strength of the dispersion to be the essential property deciding the regularity of the flow map, with stronger dispersion giving greater regularity. Theorem \ref{thm: line} shows that this is the case, for while the restriction $\gamma<2$ in Theorem \ref{thm: line} appears in the proof from our construction of a specific approximate solution to \eqref{eq: whitham}, it is, in fact, optimal. When $\gamma=2$ our assumptions includes the KdV equation for which the flow map is known to be locally Lipschitz in $H^s(\R)$ for $s>-\frac{3}{4}$ \cite{Kenig1996abe}, meaning that Theorem \ref{thm: line} is not true in this case. For $0\leq \gamma<2$ it was proved in \cite{Molinet2001ipi} that for $m(\xi)$ such that $p(\xi)=\xi m(\xi)$ is differentiable and satisfying $|p'(\xi)|\leq |\xi|^\gamma$ for $0\leq \gamma<2$, which implies that the the assumptions of Theorem \ref{thm: line} are satisfied, the flow map cannot be $C^2$ in $H^s(\R)$ for any $s\in \R$. Our findings are in agreement with, and improves upon, these resuls. In the period case \eqref{eq: whitham} has, in a sense, no dispersive effect as it is invariant under the transformation
\begin{equation*}
u(x,t)\mapsto v(x,t)=u(x-t\omega,t)+\omega.
\end{equation*}
Having no restriction on $p$ in Theorem \ref{thm: periodic} is therefore perfectly in line with the notion that the strength of the dispersion is the decisive factor for the regularity of the flow map. To avoid this situation, one often considers initial data having zero mean $\int_\T f(x)\dx=\widehat{f}(0)$. In this case the flow map of the KdV equation is known to be Lipschitz continuous in $H^s(\T)$ for $s\geq 0$ \cite{Bourgain1993ftr} and thus Theorem \ref{thm: periodic} fails for $p=2$. In fact, it fails for the KdV in $H^s(\T)$ for $s>-\frac{1}{2}$ \cite{Kenig1996abe}. We will, however, not consider this case here and make no assumption on the mean of the initial data in the periodic case. 

We will prove Theorems \ref{thm: periodic} and \ref{thm: line} using a method based on \cite{Koch2005nwi}, where nonuniform dependence on initial data was established for the BO equation on $\R$, describing the effect of a low-frequency perturbation on a high-frequency wave. In \cite{Molinet2007gwp} the proof of \cite{Koch2005nwi} for the BO equation on the line is adapted to the simpler periodic case for the fKdV equation. For the periodic case, the arguments are easily extended to operators with more general symbols $m$, and the proof we present for Theorem \ref{thm: periodic} is a straightforward extension of that of \cite{Molinet2007gwp} and \cite{Koch2005nwi}. Non-uniform continuity for fractional KdV equations on the line has not been proved for general order $0\leq \alpha<2$, but as the symbol $m(\xi)=|\xi|^\alpha$ is homogeneous the equation enjoys scaling properties similar to those of the BO equation and the procedure of \cite{Koch2005nwi} should therefore be applicable without too much difficulty. The Whitham equation however, or indeed any equation with inhomogeneous symbol $m$, does not share these properties, and different argumentation is therefore required (see Section \ref{sec: line} and in particular Proposition \ref{prop: approx scaling}).

Our paper is structured as follows: Section \ref{sec: preliminaries} is devoted to some preliminary results on local well-posedness, existence time and energy estimates that are crucial ingredients in the proofs of Theorems \ref{thm: periodic} and \ref{thm: line}. In Sections \ref{sec: periodic} and \ref{sec: line} we prove Theorem \ref{thm: periodic} and Theorem \ref{thm: line}, respectively.

\section{Preliminaries}
In this section we state results on the existence of solutions to equation \eqref{eq: whitham} with initial data \eqref{eq: whitham id} and estimates on the existence time and $H^s$-norm of the solutions. All the results in this section hold equally on $\T$ and on $\R$, and we will denote by $H^s$ either $H^s(\T)$ or $H^s(\R)$. The main result is the following:
\label{sec: preliminaries}
\begin{theorem}[\cite{Ehrnstrom2016olw}]
\label{thm: local well-posedness}
Assume that $m$ is even, $m\in L_{loc}^\infty(\R)$ and that $|m(\xi)\lesssim |\xi|^p$ for some $p$ when $|\xi|>1$. Then, for $s>\frac{3}{2}$ and $u_0\in H^s$ there is a maximal $T>0$ depending only on $\|u_0\|_{H^s}$, and a unique solution $u$ to \eqref{eq: whitham}-\eqref{eq: whitham id} in the class $C\left([0,T);H^s\right)$. The solution depends continuously on the initial data, i.e. the map $u_0 \mapsto u(t)$ is continuous from $H^s$ to $C\left([0,T);H^s\right)$.
\end{theorem}
Moreover, we have the following lower bound for the existence time $T$ and relationship between the $H^s$ norm of the solution $u$ at time $t$ and the $H^s$ norm of the initial data:
\begin{lemma}
\label{thm: existence time}
Let $s>\frac{3}{2}$. If $u$ is the solution to \eqref{eq: whitham} with initial data $u_0\in H^s$ described in Theorem \ref{thm: local well-posedness}, then there exists a constant $c_s$, depending only on $s$, such that
\begin{equation}
\label{eq: bound on solution norm}
\|u(t)\|_{H^s}\leq \frac{\|u_0\|_{H^s}}{1-t c_s\|u_0\|_{H^s}}.
\end{equation}
In particular, the maximal existence time $T$ in Theorem \ref{thm: local well-posedness} satisfies
\begin{equation*}
T\geq \frac{1}{c_s \|u_0\|_{H^s}}.
\end{equation*}
\end{lemma}
\begin{remark}
\label{rem: remark}
Lemma \ref{thm: existence time} is a typical result for equations of the form \eqref{eq: whitham} and can be proved by standard arguments, but we give a proof here for the sake of completeness. We prove Lemma \ref{thm: existence time} on the line. How to extend the proof to the periodic case should be clear. The proof follows the proof of Proposition 1 in \cite{Himonas2009nud}, an equivalent result for the CH equation, but is in fact simpler due to the the operator $L$ being skew-symmetric and linear.
\end{remark}
In order to prove Lemma \ref{thm: existence time}, we introduce the operators $\Lambda^s$ defined by
\begin{equation*}
\widehat{\Lambda^s f}(\xi)=(1+\xi^2)^{s/2}\widehat{f}(\xi), \quad s\in \R.
\end{equation*}
Note that $\|\Lambda^s f\|_{L^2(\R)}=\|f\|_{H^s(\R)}$.
\begin{proof}
The proof relies on the following differential inequality for the solution $u$ that we will establish:
\begin{equation}
\label{eq: differential ineq}
\frac{1}{2}\frac{\mathrm{d}}{\mathrm{d}t}\|u(t)\|_{H^s(\R)}^2\leq c_s \|u(t)\|_{H^s(\R)}^3.
\end{equation}
Solving \eqref{eq: whitham} for $\partial_t u$, we get
\begin{equation*}
\partial_t u=-u\partial_x u-L(\partial_x u).
\end{equation*}
In order to make all the terms be in $H^s(\R)$, we mollify, which we write as
\begin{equation*}
J_\varepsilon f=j_\varepsilon \ast f.
\end{equation*}
Thus we consider the equation
\begin{equation}
\label{eq: mollified eq.}
\partial_t J_\varepsilon u=-J_\varepsilon(u\partial_x u)-L(\partial_x J_\varepsilon u),
\end{equation}
where writing $L(\partial_x J_\varepsilon u)$ in the last term is justified as follows: Firstly, writing $L(u)$ as a convolution $\F^{-1}(m(\xi))\ast u$, associativity and commutativity of convolution gives that $J_\varepsilon$ and $L$ commutes,
\begin{align*}
J_\varepsilon L(\partial_x u) & =j_\varepsilon\ast (\F^{-1}(m(\xi))\ast \partial_x u)=\F^{-1}(m(\xi))\ast(j_\varepsilon \ast \partial_x u) \\
&=L(J_\varepsilon \partial_x u).
\end{align*}
Secondly, it can easily be shown that $J_\varepsilon \partial_x u=\partial_x J_\varepsilon u$ using integration by parts.

Applying the operator $\Lambda^s$ to both sides of \eqref{eq: mollified eq.}, then multiplying the resulting equation by $\Lambda^s(J_\varepsilon u)$ and integrating it for $x\in \R$ gives
\begin{equation}
\label{eq: differential norm equation}
\frac{1}{2}\frac{\mathrm{d}}{\mathrm{d}t}\|J_\varepsilon u(t)\|_{H^s(\R)}^2=-\int_{\R} \Lambda^s(J_\varepsilon(u \partial_x u))\Lambda^s(J_\varepsilon u)\dx-\int_{\R}\Lambda^s\left( L(\partial_x J_{\varepsilon}u)\right) \Lambda^s(J_\varepsilon u)\dx.
\end{equation}
First we consider the last term on the right hand side:
\begin{align*}
\int_{\R}\Lambda^s\left( L(\partial_x J_{\varepsilon}u)\right) \Lambda^s(J_\varepsilon u)\dx &= \int_{\R}(1+\xi^2)^{s/2}\overline{\widehat{J_\varepsilon u}}(\xi)(1+\xi^2)^{s/2} m(\xi) \widehat{\partial_x J_\varepsilon u}(\xi)\dxi \\
& = \int_{\R} (1+\xi^2)^s m(\xi) i\xi \widehat{J_\varepsilon u}(-\xi)\widehat{J_\varepsilon u}(\xi)\dxi \\
&=0,
\end{align*}
where the last inequality follows from $m$ being even. For the first term on the right hand side of \eqref{eq: differential norm equation}, we know from the proof of Proposition 1 in \cite{Himonas2009nud} that
\begin{equation*}
|\int_{\R} \Lambda^s(J_\varepsilon(u \partial_x u))\Lambda^s(J_\varepsilon u)\dx|\leq c_s\|\partial_x u\|_{\infty}\|u\|_{H^s(\R)}^2
\end{equation*}
(the proof relies on commutator estimates for the operators $\Lambda^s$).
Thus we have that
\begin{equation*}
\frac{1}{2}\frac{\mathrm{d}}{\mathrm{d}t}\|J_\varepsilon u(t)\|_{H^s(\R)}^2\leq c_s \|\partial_x u\|_{L^\infty}\|u\|_{H^s(\R)}^2.
\end{equation*}
Integrating from $0$ to $t$ on both sides, we get
\begin{equation*}
\frac{1}{2}\|J_\varepsilon u(t)\|_{H^s(\R)}^2-\frac{1}{2}\|J_\varepsilon u(0)\|_{H^s(\R)}^2\leq c_s \int_0^t \|\partial_x u(\tau)\|_{\infty}\|u(\tau)\|_{H^s(\R)}^2 \dt,
\end{equation*}
and letting $\varepsilon\rightarrow 0$, we have that
\begin{equation*}
\frac{1}{2}\|u(t)\|_{H^s(\R)}^2-\frac{1}{2}\|u(0)\|_{H^s(\R)}^2\leq c_s \int_0^t \|\partial_x u(\tau)\|_{L^\infty}\|u(\tau)\|_{H^s(\R)}^2 \,\mathrm{d}\tau.
\end{equation*}
From this we deduce that
\begin{equation}
\label{eq: basis for gronwall}
\frac{1}{2}\frac{\mathrm{d}}{\mathrm{d}t}\|u(t)\|_{H^s(\R)}^2\leq c_s \|\partial_x u(t)\|_{L^\infty}\|u(t)\|_{H^s(\R)}^2.
\end{equation}
Since $s>\frac{3}{2}$, the Sobolev embedding $H^{s-1}(\R)\hookrightarrow L^\infty(\R)$ holds and we thus get \eqref{eq: differential ineq}. Now let $y(t)=\|u(t)\|_{H^s(\R)}^2$. Then \eqref{eq: differential ineq} implies
\begin{equation*}
\frac{1}{2} y^{-3/2}\frac{\mathrm{d}y}{\mathrm{d}t}\leq c_s.
\end{equation*}
Integrating from $0$ to $t$ gives
\begin{equation*}
\frac{1}{\sqrt{y(0)}}-\frac{1}{\sqrt{y(t)}}\leq c_s t,
\end{equation*}
and we obtain \eqref{eq: bound on solution norm}. From \eqref{eq: bound on solution norm} we immediately get that $\|u(t)\|_{H^s(\R)}$ is finite when $t<\frac{1}{c_s \|u_0\|_{H^s(\R)}}$, and thus we get the lower bound on the maximal existence time $T$.
\end{proof}

We also have energy estimates for arbitrary Sobolev norm. The statements below are rather rough, as we are not interested in optimizing the constants for which the inequalities are true.
\begin{corollary}
\label{cor: energy estimate}
Let $s>\frac{3}{2}$. Given $u_0\in H^s$, let $u$ be the corresponding solution. Then, for any $T_0<(c_s\|u_0\|_{H^s})^{-1}$ and all $t\in [0,T_0]$, one has
\begin{equation}
\|u(t)\|_{H^r}\lesssim \exp(C t\|u_0\|_{H^s})\|u_0\|_{H^r},
\end{equation}
for all $r>0$, for some constant $C$ depending only on $r$ and distance between $(c_s\|u_0\|_{H^s})^{-1}$ and $T_0$.
\end{corollary}
\begin{proof}
Note that in the arguments establishing \eqref{eq: basis for gronwall} in the proof of Lemma \ref{thm: existence time}, it was nowhere used that $s$ in the order of the Sobolev norm was the same $s$ as in the statement of the Lemma, so \eqref{eq: basis for gronwall} holds for any $r>0$ in place of $s$. Thus
\begin{equation*}
\frac{\mathrm{d}}{\mathrm{d}t}\|u(t)\|_{H^r}\leq c_r \|\partial_x u(t)\|_{L^\infty}\|u(t)\|_{H^r}
\end{equation*}
for any $r>0$. From Gr\"onwall's inequality, Sobolev embeddings and \eqref{eq: bound on solution norm} we then get that
\begin{align*}
\|u(t)\|_{H^r} & \leq \exp(c_r \int_0^t \|\partial_x u(\tau)\|_{L^\infty}\, \mathrm{d}\tau)\|u_0\|_{H^r} \\
& \leq C_1 \exp(c_r \int_0^t\|u(\tau)\|_{H^s}\, \mathrm{d}\tau)\|u_0\|_{H^r} \\
& \leq C_1\exp(c_r \int_0^t C_2\|u_0\|_{H^s}\, \mathrm{d}\tau)\|u_0\|_{H^r} \\
& = C_1\exp(c_r C_2 t\|u_0\|_{H^s})\|u_0\|_{H^r}
\end{align*}
for all $r>0$ and $t\in [0, T_0]$, where $C_1$ is an embedding constant and $C_2>0$ depends only on the difference $(c_s\|u_0\|_{H^s})^{-1}-T_0$.

\end{proof}
\section{The periodic case}
\label{sec: periodic}
This section is devoted to proving lack of uniform continuity for the flow map of equation \eqref{eq: whitham} on $\T$. That is, we will prove Theorem \ref{thm: periodic}. This will be done in two steps. First, we construct two sequences of approximate solutions in $H^s(\T)$ that converge to the same limit at time $0$, while remaining bounded apart at any later time. Then we show that the approximate solutions are sufficiently close to real solutions, thereby establishing lack of uniform continuity. The proof is based on \cite{Koch2005nwi} and \cite{Molinet2007gwp}.

The approximate solutions consist of a low-frequency term and a high-frequency term and are constructed as follows: For $\omega\in \R$ and $n\in \N$, we set
\begin{equation*}
u_n^\omega (x,t)=\omega n^{-1}+n^{-s}\cos(-n m(n) t+nx-\omega t).
\end{equation*}
By direct calculation, one can show that for $n\in \N$ and $\alpha\in\R$,
\begin{equation}
\label{eq: sin(n) norm}
\|\sin(nx-\alpha)\|_{H^\sigma(\T)}\simeq n^\sigma,
\end{equation}
and similarly for cosine as well. Thus, for $\omega$ bounded, we have
\begin{equation*}
\|u_n^\omega (\cdot,t)\|_{H^s(\T)}\simeq 1, \,\, \text{for all} \,\, t\in \R, \, n\in N.
\end{equation*}
In particular, $u_n^\omega\in H^s(\T)$ for all $n\in \N$ and all $t\in \R$ and the $H^s(\T)$ norm is bounded above uniformly in $n\in \N$.

The next lemma measures how far away the functions $u_n^\omega$ are from solving equation \eqref{eq: whitham} in the spaces $H^\sigma(\T)$:
\begin{lemma}
\label{lem: error estimate}
Set
\begin{equation}
\label{eq: expression for E}
E =\partial_t u_n^\omega+u_n^\omega \partial_x u_n^\omega+L(\partial_x u_n^\omega),
\end{equation}
the error of $u_n^\omega$ as an approximate solution to \eqref{eq: whitham}. Then, for $\sigma\in \R$, the error $E$ satisfies
\begin{equation*}
\|E\|_{H^\sigma(\T)}\lesssim n^{-2s+1+\sigma}.
\end{equation*}
\end{lemma}
\begin{proof}
By straightforward calculations, we find
\begin{align*}
\partial_t u_n^\omega (x,t) &=n^{-s}\left(nm(n)+\omega\right)\sin(-nm(n)t+nx-\omega t), \\
\partial_x u_n^\omega (x,t) &=-n^{-s+1}\sin(-nm(n)t+nx-\omega t), \\
L(\partial_x u_n^\omega (x,t)) &=-n^{-s+1}m(n)\sin(-nm(n)t+nx-\omega t).
\end{align*}
Inserting $u_n^\omega (x,t)$ into \eqref{eq: whitham} and using the above equalities, we get the following expression for the error:
\begin{align*}
E &=\partial_t u_n^\omega+u_n^\omega \partial_x u_n^\omega+L(\partial_x u_n^\omega) \\
&=-n^{-2s+1}\sin(-nm(n)t+nx-\omega t)\cos(-nm(n)t+nx-\omega t) \\
&=-\frac{1}{2}n^{-2s+1} \sin[2(-nm(n)t+nx-\omega t)].
\end{align*}
The statement now follows from \eqref{eq: sin(n) norm}.
\end{proof}

\begin{lemma}
\label{lem: distance at time t}
For $n\gg 1$ and $t\geq 0$,
\begin{equation*}
\|u_n^{1}(\cdot, t)-u_n^{-1}(\cdot,t)\|_{H^s(\T)}\gtrsim \sin(t).
\end{equation*}
Moreover
\begin{equation*}
\|u_n^{1}(\cdot, 0)-u_n^{-1}(\cdot,0)\|_{H^s(\T)}\rightarrow 0, \,\, \text{as} \,\, n\rightarrow \infty.
\end{equation*}
\end{lemma}
\begin{proof}
Using the basic trigonometric identity $\cos(\alpha\pm \beta)=\cos(\alpha)\cos(\beta)\mp \sin(\alpha)\sin(\beta)$ with $\alpha=-nm(n)t+nx$ and $\beta=t$, and \eqref{eq: sin(n) norm} we get
\begin{align*}
\| & u_n^{1}(\cdot, t)  -u_n^{-1}(\cdot,t)\|_{H^s(\T)} \\
&= \|2n^{-1}+n^{-s}\left[\cos(-nm(n)t+n\cdot-t)-\cos(-nm(n)t+n\cdot+ t)\right]\|_{H^s(\T)} \\
&= \|2n^{-1}+2n^{-s}\sin(-nm(n)t+n\cdot)\sin(t)\|_{H^s(\T)} \\
& \gtrsim 2n^{-s}|\sin(t)|\|\sin(-nm(n)t+n\cdot)\|_{H^s(\T)} -\frac{2}{n} \\
& \simeq |\sin(t)|-\frac{1}{n}.
\end{align*}
This proves the first statement. Setting $t=0$ in the calculations above, it is plain to see that the second statement also holds.
\end{proof}

Now we show that the approximate solutions $u_n^{\omega}$ are sufficiently close to real solutions $v_n^{\omega}$ of \eqref{eq: whitham} for $n\gg 1$.
\begin{lemma}
\label{lem: approx exact difference s<2}
Let 
$v_n^{\omega}(x,t)$ be the $H^s(\T)$ solution to the Cauchy problem
\begin{align*}
& \partial_t v_n^{\omega}+v_n^{\omega}\partial_x v_n^{\omega}+L(\partial_x v_n^{\omega})=0, \\
& v_n^{\omega}(x,0)=\omega n^{-1}+n^{-s}\cos(nx).
\end{align*}
That is, $v_n^{\omega}$ is a solution to equation \eqref{eq: whitham} with initial data given by $u_n^{\omega}$ evaluated at time $t=0$. Then the following holds:
\begin{itemize}
\item[(i)] If $s>\frac{3}{2}$, there exists $T_0>0$ independent of $n$ such that for any $k>s$,
\begin{equation*}
\|u_n^{\omega}(t)-v_n^{\omega}(t)\|_{H^s(\T)} \lesssim n^{(1-s)(1-\frac{s}{k})}, \,\, 0\leq t\leq T_0 , \,\, n\gg 1.
\end{equation*}
\item[(ii)] If $0<s\leq \frac{3}{2}$, there exists $0\leq T_n\lesssim n^{s-\sigma}$, where $\sigma>\frac{3}{2}$ can be arbitrarily close to $\frac{3}{2}$, such that
\begin{equation*}
\|u_n^{\omega}(t)-v_n^{\omega}(t)\|_{H^s(\T)}\lesssim n^{-(1/2)(1-s/k)}
\end{equation*}
for any $k>\frac{3}{2}$ and $0\leq t\leq T_n$.
\end{itemize}

\end{lemma}
\begin{proof}
We prove (i) first; that is, we assume $s>\frac{3}{2}$ so that the Cauchy problem is locally well-posed in $H^s(\T)$. As $\|u_n^{\omega}\|_{H^s(\T)}\simeq 1$ for all $n\in \N$, Theorem \ref{thm: local well-posedness} and Lemma \ref{thm: existence time} guarantees the existence of $v_n^{\omega}\in H^s(\T)$ up to some time $T\simeq 1$ that can be considered independent of $n$. Letting $T_0$ be strictly smaller than the $T$ given by Lemma \ref{thm: existence time}, for instance $T_0=\frac{1}{2}T$, we have that
\begin{equation*}
\|v_n^{\omega}(t)\|_{H^s(\T)}\lesssim \|u_n^{\omega}\|_{H^s(\T)}
\end{equation*}
for all $0\leq t\leq T_0$.

Set $w=u_n^{\omega}-v_n^{\omega}$. Straightforward calculations, using the expression \eqref{eq: expression for E} for $E$ and that $v_n^{\omega}$ is an exact solution to \eqref{eq: whitham}, show that $w$ solves the initial value problem
\begin{align}
& \partial_t w=E+w\partial_x w-\partial_x (w u_n^\omega)-L(\partial_x w) \label{eq: equation for w} \\ & w(\cdot,0)=0. \nonumber
\end{align}
Multiplying by $w$ on both sides of \eqref{eq: equation for w}, we see that
\begin{align*}
\frac{1}{2}\frac{\mathrm{d}}{\mathrm{d}t}\|w(t)\|_{L^2(\T)}^2 = &\int_{\T}w E\dx \\
& +\int_{\T}w^2 \partial_x w\dx \\
& -\int_\T w\partial_x (w u_n^\omega) \dx \\
& -\int_{\T}wL(\partial_x w)\dx.
\end{align*}
Using Parseval's identity and that $m(\xi)$ is even, we see that the last integral vanishes.
The first term on the right-hand side is easily estimated by H\"older's inequality:
\begin{equation*}
\left| \int_{\T}w E\dx\right|\leq \|E\|_{L^2(\T)}\|w\|_{L^2(\T)}.
\end{equation*}
The second term is easily seen to vanish by writing $w^2\partial_x w=\partial_x (w^3)$. For the third term we use integration by parts and H\"older's inequality:
\begin{equation*}
\left| \int_\T w\partial_x (w u_n^\omega) \dx\right| = \frac{1}{2}\left| \int_\T u_n^\omega \partial_x (w^2) \dx\right| =\frac{1}{2}\left| \int_\T  w^2\partial_x u_n^\omega\dx\right|\lesssim \|\partial_x u_n^\omega\|_{L^\infty} \|w\|_{L^2(\T)}^2.
\end{equation*}
Combining these estimates we get the following inequality:
\begin{equation*}
\frac{1}{2}\frac{\mathrm{d}}{\mathrm{d}t}\|w(t)\|_{L^2(\T)}^2 \lesssim \|E\|_{L^2(\T)}\|w\|_{L^2(\T)}+\|\partial_x u_n^\omega\|_{L^\infty}\|w\|_{L^2(\T)}^2.
\end{equation*}
From the definition of $u_n^\omega (x,t)$ it follows that $\|\partial_x u_n^{\omega}(t)\|_{L^\infty}\lesssim n^{-s+1}$, and using Lemma \ref{lem: error estimate} we then conclude that
\begin{equation*}
\frac{1}{2}\frac{\mathrm{d}}{\mathrm{d}t}\|w(t)\|_{L^2(\T)}^2 \lesssim n^{-2s+1}\|w\|_{L^2(\T)}+n^{-s+1}\|w\|_{L^2(\T)}^2,
\end{equation*}
which implies that
\begin{equation}
\label{eq: diff inequality}
\frac{\mathrm{d}}{\mathrm{d}t}\|w(t)\|_{L^2(\T)}\lesssim n^{-s+1}\|w\|_{L^2(\T)}+n^{-2s+1}.
\end{equation}
Recalling that $w(\cdot,0)=0$, we conclude that
\begin{equation}
\label{eq: estimate sigma}
\|u_n^{\omega}(t)-v_n^{\omega}(t)\|_{L^2(\T)}\lesssim n^{-2s+1}
\end{equation}
for all $0\leq t\leq T_0$.

As $v_n^{\omega}$ is a solution to \eqref{eq: whitham}, Corollary \ref{cor: energy estimate} implies that for $k>s$,
\begin{align*}
\|v_n^\omega(t)\|_{H^k(\T)} \lesssim n^{k-s},
\end{align*}
for $t\in [0,T_0]$. We thereby get the "rough" estimate
\begin{align}
\|u_n^{\omega}(t)-v_n^{\omega}(t)\|_{H^k(\T)} & \leq \|u_n^{\omega}(t)\|_{H^k(\T)}+\|v_n^{\omega}(t)\|_{H^k(\T)} \nonumber \\
&\lesssim \|u_n^{\omega}(t)\|_{H^k(\T)}+\|u_n^{\omega}(0)\|_{H^k(\T)} \nonumber \\
& \lesssim n^{-s+k} \label{eq: rough estimate k}
\end{align}
for $k>s$ and $0\leq t\leq T_0$. Interpolating between \eqref{eq: estimate sigma} and \eqref{eq: rough estimate k} for $k>s$, we get
\begin{align*}
\|u_n^{\omega}(t)-v_n^{\omega}(t)\|_{H^s(\T)} & \leq \|u_n^{\omega}(t)-v_n^{\omega}(t)\|_{L^2(\T)}^{1-s/k} \|u_n^{\omega}(t)-v_n^{\omega}(t)\|_{H^k(\T)}^{s/k} \\
& \lesssim n^{(1-s)(1-\frac{s}{k})}.
\end{align*}
This proves part (i).

Now we turn to the case where $0<s\leq \frac{3}{2}$. As $\|u_n^\omega\|_{H^{\sigma}(\T)}\lesssim n^{\sigma-s}$, Theorem \ref{thm: local well-posedness} and Lemma \ref{thm: existence time} imply that $v_n^\omega\in H^{\sigma}(\T)$ exists and satisfies $\|v_n^\omega\|_{H^{\sigma}(\T)}\lesssim n^{\sigma-s}$ for $0\leq t\leq T\simeq n^{s-\sigma}$ for any $\sigma>\frac{3}{2}$. Taking $k>\frac{3}{2}$, we get that the estimate \eqref{eq: rough estimate k} holds for $0\leq t\leq T\simeq n^{s-\sigma}$. Moreover, \eqref{eq: diff inequality} still holds and using Gr\"onwall's inequality we conclude that
\begin{equation}
\label{eq: estimate sigma lower s}
\|u_n^{\omega}(t)-v_n^{\omega}(t)\|_{L^2(\T)}\lesssim n^{-2s+1} T
\end{equation}
for $0\leq T\lesssim n^{s-\sigma}$. Interpolating between \eqref{eq: estimate sigma lower s} and \eqref{eq: rough estimate k} for $k>\frac{3}{2}$, we get that
\begin{equation*}
\|u_n^{\omega}(t)-v_n^{\omega}(t)\|_{H^s(\T)}\lesssim n^{-(1/2)(1-s/k)}
\end{equation*}
for $0\leq t\leq T\simeq n^{s-\sigma}$. 
\end{proof}

We are now able to prove Theorem \ref{thm: periodic}:
\begin{proof}[Proof of Theorem \ref{thm: periodic}]
Let $v_n^{\omega}(x,t)$ be the $H^s(\T)$ solution to the Cauchy problem
\begin{align*}
& \partial_t v_n^{\omega}+v_n^{\omega}\partial_x v_n^{\omega}+L(\partial_x v_n^{\omega})=0, \\
& v_n^{\omega}(x,0)=\omega n^{-1}+n^{-s}\cos(nx).
\end{align*}
Assume first that $s>\frac{3}{2}$. By Lemma \ref{lem: approx exact difference s<2} we have that
\begin{align*}
\|v_n^{1}(t) &-v_n^{-1}(t)\|_{H^s(\T)} \\
& \geq  \|u_n^{1}(t)-u_n^{-1}(t)\|_{H^s(\T)}-\|u_n^{1}(t)-v_n^{1}(t)\|_{H^s(\T)} -\|u_n^{-1}(t)-v_n^{-1}(t)\|_{H^s(\T)} \\
& \gtrsim \|u_n^{1}(t)-u_n^{-1}(t)\|_{H^s(\T)}-n^{(1-s)(1-\frac{s}{k})}.
\end{align*}
As $(1-s)(1-\frac{s}{k})<0$ for $s>1$ and $k>s$, Lemma \ref{lem: distance at time t} then implies that
\begin{equation*}
\|v_n^{1}(t) -v_n^{-1}(t)\|_{H^s(\T)}\gtrsim |\sin(t)|
\end{equation*}
for all $0\leq t\leq T_0$ and $n\gg 1$. Moreover,
\begin{equation*}
\|v_n^{1}(0) -v_n^{-1}(0)\|_{H^s(\T)} \rightarrow 0
\end{equation*}
as $n\rightarrow \infty$. This proves part (i).

When $0<s\leq \frac{3}{2}$ the above arguments do not lead to a contradiction as the times $t$ for which they hold go to zero; we need the solutions to go apart much sooner. As noted in \cite{Molinet2007gwp}, the essential observation is that if $u(x,t)$ solves \eqref{eq: whitham} with initial data $u_0$, then $v(x,t)=u(x-\omega t,t)+\omega$ solves \eqref{eq: whitham} with initial data $u_0+\omega$, as is easily verified. The arguments in \cite{Molinet2007gwp} can be applied directly from this point, but we repeat them here or the sake of completeness. 

Let $v_n^0$ be a solution to the Cauchy problem above for $\omega=0$, and define $\tilde{v}_n^\omega(x,t):=v_n^0(x-\omega t,t)+\omega$. We pick $t_n\in [n^{-1+\varepsilon}, n^{s-\sigma}]$ for some $\varepsilon>0$ sufficiently small and set
\begin{equation*}
\omega_1=(n t_n)^{-1}\frac{\pi}{2},\quad \omega_2=-(n t_n)^{-1}\frac{\pi}{2}.
\end{equation*}
At time $t=0$, we get
\begin{equation*}
\|\tilde{v}_n^{\omega_1}(\cdot, 0) -\tilde{v}_n^{\omega_2}(\cdot,0)\|_{H^s(\T)}\simeq |\omega_1-\omega_2|\lesssim n^{-\varepsilon}\rightarrow 0
\end{equation*}
as $n\rightarrow \infty$. At $t=t_n$ we can use Lemma \ref{lem: approx exact difference s<2} (ii):
\begin{align*}
\|\tilde{v}_n^{\omega_1}(\cdot, t_n) -\tilde{v}_n^{\omega_2}(\cdot, t_n)\|_{H^s(\T)} = & \|v_n^0(\cdot-\omega_1 t_n,t_n)+\omega_1-v_n^0(\cdot-\omega_2 t_n,t_n)-\omega_2\|_{H^s(\T)} \\
\gtrsim & \|v_n^0(\cdot -\omega_1 t_n, t_n)-u_n^0(\cdot -\omega_1 t_n, t_n)\|_{H^s(\T)} \\
& +\|v_n^0(\cdot -\omega_2 t_n, t_n)-u_n^0(\cdot -\omega_2 t_n, t_n)\|_{H^s(\T)} \\
& + \|u_n^0(\cdot -\omega_1 t_n, t_n)-u_n^0(\cdot -\omega_2 t_n, t_n)\|_{H^s(\T)} \\
& - |\omega_1-\omega_2| \\
\simeq & 1+n^{-(1/2)(1-s/k)} -n^{-\varepsilon},
\end{align*}
where we calculated
\begin{align*}
\|u_n^0(x -\omega_1 t_n, t_n) & -u_n^0(x -\omega_2 t_n, t_n)\|_{H^s(\T)} \\
& =\|n^{-s}(\cos(-n m(n)t_n+nx-\pi/2)-\cos(-n m(n)t_n+nx+\pi/2))\|_{H^s(\T)} \\
& =\|2n^{-s}\sin(-n m(n)t_n+nx)\|_{H^s(\T)} \\
& \simeq 1.
\end{align*}
Taking $n\rightarrow \infty$, this concludes the proof of part (ii).
\end{proof}
\section{Non-uniform continuity on the real line}
\label{sec: line}
In this section we prove the lack of uniform continuity for the flow map of the Whitham equation \eqref{eq: whitham} on $\R$. That is, we will prove Theorem \ref{thm: line}. As in the periodic case (cf. Section \ref{sec: periodic}), Theorem \ref{thm: line} will be proven by constructing two sequences of approximate solution in $H^s(\R)$ that converge to the same limit at time $0$, while remaining bounded apart at any later time and showing that the approximate solutions are sufficiently close to real solutions. The idea for the proof is from \cite{Koch2005nwi}.

In the sequel, $\delta$ will always denote a number $1<\delta<2$ that we may choose freely and $\lambda$ will be a positive parameter. For convenience of notation we will denote $f_\lambda(x):=f(\frac{x}{\lambda^\delta})$ for functions $f: \R \rightarrow \R$ and $\lambda>0$.
The following lemma will be useful in the sequel.
\begin{lemma}[\cite{Koch2005nwi}]
\label{lem: norm varphi sin}
Let $\varphi\in \mathscr{S}(\R)$, $1<\delta<2$ and $\alpha\in \R$. Then for any $s\geq 0$ we have that
\begin{equation*}
\lim_{\lambda\rightarrow \infty} \lambda^{-\delta/2-s}\|\varphi_\lambda \cos(\lambda \cdot +\alpha)\|_{H^s(\R)}=\frac{1}{\sqrt{2}}\|\varphi\|_{L^2(\R)}.
\end{equation*}
The statement holds true also if $\cos$ is replaced by $\sin$.
\end{lemma}
Lemma \ref{lem: norm varphi sin} can be found as Lemma 2.3 in \cite{Koch2005nwi} and a proof is given there.

We construct a two-parameter family of approximate solutions $u^{\omega,\lambda}=u^{\omega,\lambda}(t,x)$, following \cite{Himonas2009nud}. Each function $u^{\omega,\lambda}$ consists of two parts;
\begin{equation*}
u^{\omega,\lambda}=u_l+u^h.
\end{equation*}
The high frequency part $u^h$ is given by
\begin{equation*}
u^h=u^{h,\omega,\lambda}(t,x)=\lambda^{-\delta/2-s}\varphi_{\lambda}(x)\cos(-\lambda m(\lambda) t+ \lambda x-\omega t),
\end{equation*}
where $\varphi$ is a $C^\infty$ function such that
\begin{equation*}
\varphi(x) = \begin{cases} 1, & \text{if } |x|< 1, \\ 0, & \text{if } |x|\geq 2. \end{cases}
\end{equation*}
To simplify the notation we set $\Phi=-\lambda m(\lambda) t+ \lambda x-\omega t$.
The low frequency part $u_l=u_{l,\omega,\lambda}(t,x)$ is a solution to the following Cauchy problem:
\begin{align}
\partial_t u_l+u_l\partial_x u_l+L (\partial_x u_l)=0, \label{eq: cauchy problem u_l} \\
u_l(0,x)=\omega \lambda^{-1}\tilde{\varphi}_\lambda (x), \nonumber
\end{align}
where $\tilde{\varphi}$ is a $C_0^\infty (\R)$ function such that
\begin{equation*}
\tilde{\varphi}(x)=1, \,\, \text{if} \,\, x\in \supp \varphi.
\end{equation*}

\begin{lemma}
\label{lem: norm u_l}
Let $s>\frac{3}{2}$ and $\lambda>0$. Then the solution $u_l$ to the Cauchy problem \eqref{eq: cauchy problem u_l} exists and is unique in $H^s(\R)$ up to some time $T\gtrsim \lambda^{1-\delta/2}$. In fact, the estimate
\begin{equation}
\label{eq: u_l Hs norm}
\|u_l(t)\|_{H^r(\R)}\lesssim \lambda^{-1+\delta/2}
\end{equation}
holds for any $r\in \R$ and all times $0\leq t\leq T_0$ for some $T_0\simeq \lambda^{1-\delta/2}$. In particular, this means that the existence time goes to $\infty$ as $\lambda\rightarrow \infty$, while the $H^r(\R)$ norm goes to $0$ for any $r\in \R$.
\end{lemma}
\begin{proof}
Clearly, $\|u_l(0)\|_{H^r(\R)}\lesssim \lambda^{-1+\delta/2}$ for any $r\in \R$. Theorem \ref{thm: local well-posedness} and Lemma \ref{thm: existence time} then imply that $u_l(t,x)\in H^s(\R)$ exists and is unique for $0\leq t<T$, for some $T\geq (c_s \|u_l(0)\|_{H^s(\R)})^{-1}\gtrsim \lambda^{1-\delta/2}$. Choosing $T_0<T$, for instance $T_0=\frac{1}{2}(c_s \|u_l(0)\|_{H^s(\R)})^{-1}\gtrsim \lambda^{1-\delta/2}$, Corollary \ref{cor: energy estimate} implies the estimate \eqref{eq: u_l Hs norm}. The constant implied in the symbol $\lesssim$ in \eqref{eq: u_l Hs norm} depends on $r$ through the constant $c_r$ (cf. the proof of Corollary \ref{cor: energy estimate}), but for any fixed $r\in \R$ the asymptotic behaviour with respect to the parameter $\lambda$ will be the same.
\end{proof}

The next lemma states that $u^{\omega,\lambda}$ almost solves equation \eqref{eq: whitham} when $\lambda \gg 1$.
\begin{lemma}
\label{lem: error estimate line}
Set
\begin{equation*}
F=\partial_t u^{\omega,\lambda}+ u^{\omega,\lambda}\partial_x u^{\omega,\lambda}+L(\partial_x u^{\omega,\lambda}).
\end{equation*}
If $\delta\in (1,2)$ is chosen such that $\max \lbrace 1, \gamma\rbrace <\delta <2$, where $0\leq \gamma<2$ is as in the statement of Theorem \ref{thm: line}, then
\begin{equation*}
\|F(t)\|_{L^2(\R)}\lesssim \lambda^{-s-\varepsilon},
\end{equation*}
for some $\varepsilon>0$ and all $0\leq t\leq T_0$ where $T_0$ is as in Lemma \ref{lem: norm u_l}.
\end{lemma}
In order to prove Lemma \ref{lem: error estimate line}, we will make use of the following preposition which states that for a low-frequency solution to equation \eqref{eq: whitham}, there is a scaling in time and space such that the rescaled solution almost remains a solution:
\begin{proposition}
\label{prop: approx scaling}
Let $u$ be the solution to 
\begin{align*}
\partial_t u+u\partial_x u+L (\partial_x u)=0, \nonumber \\
u(0,x)=\omega \lambda^{-1}\tilde{\varphi}(x),
\end{align*}
and set $v(t,x)=u(\lambda^{-\delta}t, \lambda^{-\delta} x)$. Then $v$ is "almost" a solution to \eqref{eq: cauchy problem u_l} in the sense that
\begin{equation*}
u_l=v+\bigO(\lambda^{-1-\delta/2})
\end{equation*}
in the $L^2(\R)$ norm, and for any $r>3/2$ and $0<k<r$,
\begin{equation*}
u_l=v+\bigO(\lambda^{-1-\delta/2+\delta k/r})
\end{equation*}
in $H^k(\R)$ for all $0\leq t\leq T$, for some $T>T_0$ for $T_0$ as in Lemma \ref{lem: norm u_l}.
\end{proposition}

\begin{proof}
Clearly, as $\tilde{\varphi}\in C_0^\infty(\R)$, $\|u(0)\|_{H^r(\R)}\lesssim \lambda^{-1}$ for all $r\in \R$. It then follows from Theorem \ref{thm: local well-posedness}, Lemma \ref{thm: existence time} and Corollary \ref{cor: energy estimate} that 
\begin{equation}
\label{eq: Hr norm u}
\|u(t)\|_{H^r(\R)}\lesssim \|u(0)\|_{H^r(\R)}\lesssim \lambda^{-1}
\end{equation}
for all $t\in [0,T_0]$ and all $r>0$. In fact, this holds for $0\leq t\leq \lambda^{\delta/2}T_0$.

Consider now a "long wave" $v(t,x)=u(\lambda^{-\delta}t, \lambda^{-\delta} x)$. Then $v(0,x)=u_l(0,x)$. By simply adding and subtracting, we get that
\begin{align*}
L(\partial_x v)(x)  = &\lambda^\delta \int_{\R}\mathrm{e}^{\mathrm{i}x\xi} m(\xi)\mathrm{i}\xi \widehat{u}(\lambda^\delta \xi)\dxi \nonumber \\
= & \lambda^{-\delta}\int_{\R}\mathrm{e}^{\mathrm{i}x\lambda^{-\delta}y} m(\lambda^{-\delta}y)\mathrm{i}y \widehat{u}(y)\dy \nonumber \\
= & \lambda^{-\delta}\int_{\R}\mathrm{e}^{\mathrm{i}x\lambda^{-\delta}y}m(y)\mathrm{i}y \widehat{u}(y)\dy \nonumber \\
& +\lambda^{-\delta}\int_{\R}\mathrm{e}^{\mathrm{i}x\lambda^{-\delta}y}(m(\lambda^{-\delta} y)-m(y))\mathrm{i}y \widehat{u}(y)\dy \nonumber \\
= & \lambda^{-\delta} L(\partial_x u)(\lambda^{-\delta} x)+E. 
\end{align*}
Using this expression for $E$ and that $u$ is an exact solution to \eqref{eq: whitham}, we see that
\begin{align*}
\partial_t v+v\partial_x v+L(\partial_x v)= & \lambda^{-\delta}(\partial_t u+u\partial_x u+L(\partial_x u)) + E\\
= & E.
\end{align*}
Now we estimate the error $E$. Using the basic identity $\F [f(a \cdot)](\xi)=|a|^{-1}\widehat{f}(a^{-1}\xi)$ for any constant $a\neq 0$ and $f\in L^2(\R)$; that $|m(\xi)|\lesssim 1+|\xi|^\gamma$ for all $\xi\in \R$; and making a change of variables, we get
\begin{align*}
\|E\|_{L^2(\R)} & \leq \lambda^{\delta}\left(\int_{\R} |m(\xi)-m(\lambda^\delta \xi)|^2 \xi^2 |\widehat{u}(\lambda^{\delta}\xi)|^2\dxi\right)^{1/2} \\
& \leq \lambda^{-\delta/2}\left(\int_{\R} |m(\lambda^{-\delta}y)-m(y)|^2 y^2 |\widehat{u}(y)|^2\dy\right)^{1/2} \\
&\leq \lambda^{-\delta/2}\|u\|_{H^{1+\gamma}(\R)} \\
& \lesssim \lambda^{-1-\delta/2}.
\end{align*}
Now we set $w=v-u_l$. Then $w$ solves the equation
\begin{align*}
\partial_t w-w\partial_x w+v\partial_x w+w\partial_x v+L(\partial_x w) & = E, \\
w(0,x) & = 0.
\end{align*}
We want to estimate $\|w\|_{L^2(\R)}$. This is obviously equal to $0$ at time $t=0$, and we therefore estimate the change in time:
\begin{equation*}
\frac{1}{2}\frac{\mathrm{d}}{\mathrm{d}t}\|w(t)\|_{L^2(\R)}^2=\int_{\R}wE+w^2\partial_x w-w v\partial_x w-w^2\partial_x v-wL(\partial_x w) \dx.
\end{equation*}
As $m$ is even, we readily calculate that the last term vanishes:
\begin{equation*}
\int_{\R}wL(\partial_x w) \dx=\int_{\R} \widehat{w}(-\xi)m(\xi)\mathrm{i} \xi \widehat{w}(\xi)\dxi=0.
\end{equation*}
Clearly, the term $w^2\partial_x w=\frac{1}{3}\partial_x w^3$ also vanishes upon integrating, and using integration by parts we find
\begin{equation*}
\int_{\R}w v\partial_x w+w^2\partial_x v\dx=\frac{1}{2}\int_{\R}w^2\partial_x v\dx.
\end{equation*}
Thus
\begin{equation*}
\frac{1}{2}\frac{\mathrm{d}}{\mathrm{d}t}\|w(t)\|_{L^2(\R)}^2\leq \|w\|_{L^2(\R)}\|E\|_{L^2(\R)}+\|\partial_x v\|_{L^\infty}\|w\|_{L^2(\R)}^2,
\end{equation*}
and equivalently
\begin{align}
\frac{\mathrm{d}}{\mathrm{d}t}\|w(t)\|_{L^2(\R)} & \leq \|E\|_{L^2(\R)}+\|\partial_x v\|_{L^\infty}\|w\|_{L^2(\R)} \nonumber \\
& \lesssim \lambda^{-1-\delta/2}+\lambda^{-1-\delta}\|w\|_{L^2(\R)} \nonumber \\
& \lesssim \lambda^{-1-\delta/2}, \label{eq: L2 norm w}
\end{align}
where we used that $\|w\|_{L^2(\R)}=0$ at $t=0$.
For $r>3/2$ and $t\in [0, T_0]$, \eqref{eq: Hr norm u} and Lemma \ref{lem: norm u_l} gives the rough estimate
\begin{equation}
\label{eq: H^r norm w}
\|w(t)\|_{H^r(\R)}\leq \|v(t)\|_{H^r(\R)}+\|u_l(t)\|_{H^r(\R)}\lesssim \lambda^{-1+\delta/2}.
\end{equation}
Interpolating between \eqref{eq: L2 norm w} and \eqref{eq: H^r norm w} for $0<k<r$, we get
\begin{equation*}
\|w\|_{H^k(\R)}\leq \|w\|_{L^2(\R)}^{1-k/r}\|w\|_{H^r(\R)}^{k/r}\lesssim \lambda^{-1-\delta/2+\delta k/r}.
\end{equation*}
This proves the result.
\end{proof}

\begin{proof}[Proof of Lemma \ref{lem: error estimate line}]
Substituting $u^{\omega,\lambda}=u_l+u^h$ into equation \eqref{eq: whitham} we get the following expression:
\begin{align*}
F = & \partial_t u^h+u_l \partial_x u^h+u^h\partial_x u_l+u^h\partial_x u^h+L(\partial_x u^h) \\
&+\partial_t u_l +u_l\partial_x u_l+L(\partial_x u_l).
\end{align*}
Considering the fact that $u_l$ solves equation \eqref{eq: whitham} we get that the second line is zero. Computing $\partial_t u^h$, we get
\begin{equation}
\label{eq: u_t^h}
\partial_t u^h=\omega \lambda^{-\delta/2-s}\varphi_\lambda \sin(\Phi)+\lambda m(\lambda)\lambda^{-\delta/2-s}\varphi_\lambda\sin(\Phi).
\end{equation}
Since $\varphi_\lambda \tilde{\varphi}_\lambda=\varphi_\lambda$, we can write the first term on the right hand side of \eqref{eq: u_t^h} as
\begin{equation*}
\omega \lambda^{-\delta/2-s}\varphi_\lambda \sin(\Phi)=\lambda u_l(0,x)\lambda^{-\delta/2-s}\varphi_\lambda(x)\sin(\Phi).
\end{equation*}
Calculating $\partial_x u^h$, we get
\begin{equation*}
\partial_x u^h=- \lambda^{-\delta/2-s+1}\varphi_\lambda(x)\sin(\Phi)+\lambda^{-3\delta/2-s}\varphi_\lambda'(x)\cos(\Phi).
\end{equation*}
Thus
\begin{align}
\partial_t u^h+u_l \partial_x u^h = & (u_l(0,x)-u_l(t,x))\lambda^{-\delta/2-s+1}\varphi_\lambda(x)\sin(\Phi) \nonumber \\
& -u_l(t,x)\lambda^{-3\delta/2-s}\varphi_\lambda'(x)\cos(\Phi) \nonumber \\
& + \lambda m(\lambda)\lambda^{-\delta/2-s}\varphi_\lambda\sin(\Phi). \label{eq: u_t^h+u_l u_x^h}
\end{align}
Consider now the term $L(\partial_x u^h)$:
\begin{equation}
\label{eq: L u_x^h}
L(\partial_x u^h)=\lambda^{-3\delta/2 -s}L(\varphi_\lambda '\cos(\Phi))-\lambda^{-\delta/2-s+1}L(\varphi_\lambda \sin(\Phi)).
\end{equation}
Using the ancient trick of adding and subtracting, we can write
\begin{equation*}
L(\varphi_\lambda \sin(\Phi))=\left[ L, \varphi_\lambda \right] \sin(\Phi)+\varphi_\lambda L(\sin(\Phi)).
\end{equation*}
Considering $\sin(\Phi)$ as a tempered distribution, it is an easy exercise to show that $ L(\sin(\Phi))=m(\lambda)\sin(\Phi)$ for $\lambda\gg 1$, and thus the term $-\lambda^{-\delta/2-s+1}\varphi_\lambda L(\sin(\Phi))$ we get from the second term in \eqref{eq: L u_x^h} cancels out with the last term in \eqref{eq: u_t^h+u_l u_x^h}. We can therefore write the error $F$ as
\begin{align*}
F  =  & (u_l(0,x)-u_l(t,x))\lambda^{-\delta/2-s+1}\varphi_\lambda(x)\sin(\Phi) \nonumber \\
& -u_l(t,x)\lambda^{-3\delta/2-s}\varphi_\lambda'(x)\cos(\Phi) \nonumber \\
& + \lambda^{-3\delta/2 -s}L(\varphi_\lambda '\cos(\Phi)) \nonumber \\
& -\lambda^{-\delta/2-s+1}\left[ L, \varphi_\lambda \right] \sin(\Phi) \nonumber \\
& +u^h\partial_x u_l \nonumber \\
& +u^h\partial_x u^h \nonumber \\
=: & F_1+F_2+F_3+F_4+F_5+F_6. 
\end{align*}
We will estimate the terms $F_1$ to $F_6$ separately, starting with $F_1$:
\begin{align}
\|F_1\|_{L^2(\R)} & =\|(u_l(0,x)-u_l(t,x))\lambda^{-\delta/2-s+1}\varphi_\lambda(x)\sin(\Phi)\|_{L^2(\R)} \nonumber \\
& \leq \lambda^{-\delta/2-s+1}\|\varphi_\lambda(x)\sin(\Phi)\|_{L^\infty}\|u_l(0,x)-u_l(t,x)\|_{L^2(\R)} \nonumber \\
& \lesssim \lambda^{-\delta/2-s+1}\|u_l(0,x)-u_l(t,x)\|_{L^2(\R)}. \label{eq: estimate F1}
\end{align}
Obviously, at time $t=0$, $\|u_l(t,x)-u_l(0,x)\|_{L^2(\R)}=0$, and we therefore estimate the change in time:
\begin{align*}
\frac{1}{2}\frac{\mathrm{d}}{\mathrm{d}t}\|u_l(t,x) &-u_l(0,x)\|_{L^2(\R)}^2 \\
& =\frac{1}{2}\frac{\mathrm{d}}{\mathrm{d}t} \int_{\R} |u_l(t,x)|^2+|u_l(0,x)|^2-2u_l(t,x)u_l(0,x)\dx \\
& = \int_{\R} \partial_t u_l(t,x) u_l(t,x)\dx-\int_{\R} \partial_t u_l(t,x)u_l(0,x)\dx \\
& \lesssim \|\partial_t u_l(t,x)\|_{L^2(\R)}\|u_l(t,x)-u_l(0,x)\|_{L^2(\R)},
\end{align*}
which implies that
\begin{equation*}
\frac{\mathrm{d}}{\mathrm{d}t}\|u_l(t,x)-u_l(0,x)\|_{L^2(\R)}\lesssim \|\partial_t u_l(t,x)\|_{L^2(\R)}.
\end{equation*}
Solving \eqref{eq: cauchy problem u_l} for $\partial_t u_l$, we get
\begin{align}
\|\partial_t u_l(t,x)\|_{L^2(\R)} &\leq \|u_l \partial_x u_l\|_{L^2(\R)}+\|L(\partial_x u_l)\|_{L^2(\R)} \nonumber \\
& \lesssim \|u_l\|_{L^2(\R)}\|\partial_x u_l\|_{L^\infty}+\|L(\partial_x u_l)\|_{L^2(\R)}. \label{eq: partial_t u_l}
\end{align}
The first term on the right hand side is easily seen to satisfy $\|u_l\|_{L^2(\R)}\|\partial_x u_l\|_{L^\infty}\lesssim \lambda^{-2+\delta}$ for all $0\leq t\leq T_0$ by Lemma \ref{lem: norm u_l} and Sobolev embeddings. Using Proposition \ref{prop: approx scaling}, we find that
\begin{align*}
\|L(\partial_x u_l)\|_{L^2(\R)} & \lesssim \|\partial_x u_l\|_{H^\gamma (\R)} \\
& \lesssim \|\partial_x v\|_{H^\gamma (\R)} +\lambda^{-1-\delta/2+\delta(\gamma+1)/r} \\
& \lesssim \lambda^{-1-\delta/2}+\lambda^{-1-\delta/2+\delta(\gamma+1)/r}
\end{align*}
for all $0\leq t\leq T_0$. Choosing $r>2(\gamma+1)$, we conclude from \eqref{eq: partial_t u_l} that
\begin{equation*}
\frac{\mathrm{d}}{\mathrm{d}t}\|\partial_t u_l\|_{L^2(\R)}\lesssim \lambda^{-2+\delta},
\end{equation*}
and as a consequence
\begin{equation*}
\|u_l(0,x)-u_l(t,x)\|_{L^2(\R)}\lesssim \lambda^{-2+\delta}.
\end{equation*}
From \eqref{eq: estimate F1} we then get that
\begin{equation}
\label{eq: F_1}
\|F_1\|_{L^2(\R)}\lesssim \lambda^{-1+\delta/2-s}
\end{equation}
for all $0\leq t\leq T_0$.

Using Lemmas \ref{lem: norm varphi sin} and \ref{lem: norm u_l}, we readily obtain
\begin{equation}
\label{eq: F_2}
\|F_2\|_{L^2(\R)}\lesssim \lambda^{-3\delta/2-s}\|u_l\|_{L^2(\R)}\|\varphi_\lambda'(x)\cos(\Phi)\|_{L^2(\R)}\lesssim \lambda^{-\delta/2-1-s}
\end{equation}
for all $0\leq t\leq T_0$.
For $F_3$ the estimate is equally straightforward:
\begin{align*}
\|F_3\|_{L^2(\R)} & =\lambda^{-3\delta/2 -s}\|L(\varphi_\lambda '\cos(\Phi))\|_{L^2(\R)}  \\
& \lesssim \lambda^{-3\delta/2 -s}\|\varphi_\lambda '\cos(\Phi)\|_{H^\gamma (\R)}  \\
& \lesssim \lambda^{-\delta-s+\gamma}. 
\end{align*}
As $1\leq \gamma<2$, we can choose $\delta\in (1,2)$ bigger than $\gamma$ such that 
\begin{equation}
\label{eq: F_3}
\|F_3\|_{L^2(\R)}\lesssim \lambda^{-s-\varepsilon}
\end{equation}
for some $\varepsilon>0$.
This estimate, in fact, holds for all $t\in \R$.

Considering $\sin(\Phi)$ as a tempered distribution and using the symmetry of $m(\xi)$, we get
\begin{align*}
\|\left[ L, \varphi_\lambda \right] \sin(\Phi)\|_{L^2(\R)}^2 &=\int_{\R} |\widehat{\varphi_\lambda \sin(\Phi)}(m(\xi)-m(\lambda))|^2\dxi \\
& \simeq \int_{\R} |\widehat{\varphi_\lambda}(\xi-\lambda)(m(\xi)-m(\lambda))|^2\dxi \\
& =\int_{\R} |\lambda^\delta \widehat{\varphi}(\lambda^\delta (\xi-\lambda))(m(\xi)-m(\lambda))|^2\dxi.
\end{align*}
As $\widehat{\varphi}$ is a rapidly decreasing function, 
\begin{equation*}
\lim_{\lambda \rightarrow \infty}\lambda^q \widehat{\varphi}(\lambda x)\rightarrow 0 \,\, \text{for all} \,\, |x|>0 \,\, \text{and all} \,\, q>0.
\end{equation*}
Thus
\begin{equation*}
\int_{|\xi-\lambda|\geq \lambda^{-p}} |\lambda^\delta \widehat{\varphi}(\lambda^\delta (\xi-\lambda))(m(\xi)-m(\lambda))|^2\dxi=\bigO(\lambda^{-q})
\end{equation*}
for any $0<p<\delta$ and all $q>0$ when $\lambda\gg 1$. For $|\xi-\lambda|<\lambda^{-p}$ we calculate
\begin{align*}
\int_{|\xi-\lambda|< \lambda^{-p}} |\lambda^\delta \widehat{\varphi}(\lambda^\delta(\xi-\lambda))(m(\xi)-m(\lambda))|^2\dxi & = \lambda^{2\delta} \int_{|y|<\lambda^{-p}}|\widehat{\varphi}(\lambda^\delta y)|^2 |m(\lambda +y)-m(\lambda)|^2\dy \\
& \lesssim \lambda^{2\delta+2\gamma-2}\int_{|y|<\lambda^{-p}}|\widehat{\varphi}(\lambda^\delta y)|^2 |y|^2\dy \\
&=\lambda^{2\gamma-2-\delta} \int_{|y|<\lambda^{\delta-p}}|\widehat{\varphi}(y)|^2 |y|^2\dy \\
& \simeq \lambda^{2\gamma-2-\delta},
\end{align*}
where we used that $|m(\lambda+y)-m(\lambda)|\lesssim |y|\lambda^{\gamma-1}$ for $\lambda \gg 1$ and $|y|$ sufficiently small. Thus $\|\left[ L, \varphi_\lambda \right] \sin(\Phi)\|_{L^2(\R)}\lesssim \lambda^{\gamma-1-\delta/2}$, and
\begin{equation*}
\|F_4\|_{L^2(\R)}=\lambda^{-\delta/2-s+1}\|\left[ L, \varphi_\lambda \right] \sin(\Phi)\|_{L^2(\R)} \lesssim \lambda^{\gamma-s-\delta}.
\end{equation*}
By assumption, $\gamma<2$, so we may choose $\delta\in (1,2)$ such that $\delta>\gamma$ and 
\begin{equation}
\label{eq: F_4}
\|F_4\|_{L^2(\R)}\lesssim \lambda^{-s-\varepsilon}
\end{equation}
for some $\varepsilon>0$ and all $t\in\R$. The two last terms are straightforward to estimate. Using Lemmas \ref{lem: norm varphi sin} and \ref{lem: norm u_l}, we get
\begin{equation}
\label{eq: F_5}
\|F_5\|_{L^2(\R)}=\|u^h\partial_x u_l\|_{L^2(\R)}\leq \|u^h\|_{L^2(\R)}\|u_l\|_{L^\infty}\lesssim \lambda^{-s-1+\delta/2}
\end{equation}
for all $0\leq t\leq T_0$, and
\begin{align}
\|F_6\|_{L^2(\R)} & =\|u^h\partial_x u^h\|_{L^2(\R)} \nonumber \\
& \leq \lambda^{-\delta-2s+1}\|\varphi_\lambda^2 \sin(2\Phi)\|_{L^2(\R)}+\lambda^{-2\delta-2s}\|\varphi_\lambda \varphi_\lambda' \cos^2(\Phi)\|_{L^2(\R)} \label{eq: F_6}\\
& \lesssim \lambda^{-\delta/2-2s+1}+\lambda^{-3\delta/2-2s} \nonumber
\end{align}
for all $t\in \R$. The statement now follows by combining \eqref{eq: F_1}, \eqref{eq: F_2}, \eqref{eq: F_3}, \eqref{eq: F_4}, \eqref{eq: F_5} and \eqref{eq: F_6}.
\end{proof}

Now we will show that the approximate solutions $u^{\omega, \lambda}$ are arbitrarily close to exact solutions as $\lambda\rightarrow \infty$.

\begin{lemma}
\label{lem: exact approximate}
Let $s>\frac{3}{2}$ and let $u_{\omega,\lambda}$ be the $H^s(\R)$ solution to
\begin{align}
\partial_t u_{\omega,\lambda}+u_{\omega,\lambda}\partial_x u_{\omega,\lambda}+L (\partial_x u_{\omega,\lambda})=0, \nonumber \\
u_{\omega,\lambda}(0,x)=u^{\omega,\lambda}(0,x). \label{eq: cauchy problem u_omega,lambda}
\end{align}
That is, $u_{\omega,\lambda}$ is the solution to \eqref{eq: whitham} with initial data given by $u^{\omega,\lambda}$ evaluated at time $t=0$. Then there exists $T_0>0$ that can be considered independent of $\lambda$ when $\lambda\gg 1$ and $k>s$ such that
\begin{equation*}
\|u^{\omega,\lambda}-u_{\omega,\lambda}\|_{H^s(\R)}\leq \|u^{\omega,\lambda}-u_{\omega,\lambda}\|_{L^2(\R)}^{1-s/k}\|u^{\omega,\lambda}-u_{\omega,\lambda}\|_{H^k(\R)}^{s/k}\lesssim \lambda^{-\varepsilon(1-\frac{s}{k})},
\end{equation*}
for some $\varepsilon>0$, $0\leq t\leq T_0$ and $\lambda\gg 1$.
\end{lemma}
\begin{proof}
Lemmas \ref{lem: norm varphi sin} and \ref{lem: norm u_l} imply that $\|u^{\omega,\lambda}(0)\|_{H^s(\R)}\lesssim \lambda^{-1+\delta/2}+1$, and Theorem \ref{thm: local well-posedness} and Lemma \ref{thm: existence time} then imply that $u_{\omega,\lambda}\in H^s(\R)$ exists up to some time $T\simeq 1$ for $\lambda\gg 1$. Setting $T_0<T$, for instance $T_0=\frac{1}{2}T$, we have that
\begin{equation*}
\|u_{\omega,\lambda}(t)\|_{H^s(\R)}\lesssim \lambda^{-1+\delta/2}+1
\end{equation*}
for all $0\leq t\leq T_0$. We define
\begin{equation*}
v=u^{\omega,\lambda}-u_{\omega,\lambda}.
\end{equation*}
As $u_{\omega,\lambda}$ satisfies \eqref{eq: cauchy problem u_omega,lambda}, we get that $v$ satisfies
\begin{align}
\partial_t v-v\partial_x v+u^{\omega,\lambda}\partial_x v+v\partial_x u^{\omega,\lambda}+L(\partial_x v)=F, \nonumber \\
v(0,x)=0. \label{eq: caucy problem v}
\end{align}
Using \eqref{eq: caucy problem v} we can write
\begin{equation*}
\frac{1}{2}\frac{\mathrm{d}}{\mathrm{d}t}\|v\|_{L^2(\R)}^2=\int_{\R}vF+v^2\partial_x v-vu^{\omega,\lambda}\partial_x v-v^2\partial_x u^{\omega,\lambda}-vL(\partial_x v) \dx.
\end{equation*}
We readily calculate that the last term vanishes:
\begin{equation*}
\int_{\R}vL(\partial_x v) \dx=\int_{\R} \widehat{v}(-\xi)m(\xi)\mathrm{i} \xi \widehat{v}(\xi)\dxi=0.
\end{equation*}
Clearly, the term $v^2\partial_x v=\frac{1}{3}\partial_x v^3$ also vanishes upon integrating, and using integration by parts we calculate
\begin{equation*}
\int_{\R}vu^{\omega,\lambda}\partial_x v+v^2\partial_x u^{\omega,\lambda}\dx=\frac{1}{2}\int_{\R}v^2\partial_x u^{\omega,\lambda}\dx.
\end{equation*}
Thus
\begin{equation*}
\frac{1}{2}\frac{\mathrm{d}}{\mathrm{d}t}\|v\|_{L^2(\R)}^2\leq \|v\|_{L^2(\R)}\|F\|_{L^2(\R)}+\|\partial_x u^{\omega,\lambda}\|_{L^\infty}\|v\|_{L^2(\R)}^2,
\end{equation*}
and equivalently
\begin{equation*}
\frac{\mathrm{d}}{\mathrm{d}t}\|v\|_{L^2(\R)}\leq \|F\|_{L^2(\R)}+\|\partial_x u^{\omega,\lambda}\|_{L^\infty}\|v\|_{L^2(\R)}.
\end{equation*}
Straightforward calculations give
\begin{equation*}
\|\partial_x u^h\|_{L^\infty}\lesssim \lambda^{-\delta/2-s+1},
\end{equation*}
and by Sobolev embeddins and Lemma \ref{lem: norm u_l} we get
\begin{equation*}
\|\partial_x u_l\|_{L^\infty}\lesssim \lambda^{-1+\delta/2},
\end{equation*}
and thus
\begin{equation*}
\|\partial_x u^{\omega,\lambda}\|_{L^\infty}\lesssim \lambda^{-1+\delta/2}+\lambda^{-\delta/2-s+1}.
\end{equation*}
By Lemma \ref{lem: error estimate line} we then get that
\begin{equation*}
\frac{\mathrm{d}}{\mathrm{d}t}\|v\|_{L^2(\R)} \lesssim \lambda^{-s-\varepsilon}+(\lambda^{-1+\delta/2}+\lambda^{-\delta/2-s+1}) \|v\|_{L^2(\R)}.
\end{equation*}
Since $\|v(0)\|_{L^2(\R)}=0$, $-1+\delta/2<0$ and $-\delta/2-s+1<0$ for any $\delta\in(1,2)$ when $s>\frac{3}{2}$, we deduce that
\begin{equation}
\label{eq: L2 estimate v}
\|v(t)\|_{L^2(\R)} \lesssim \lambda^{-s-\varepsilon}
\end{equation}
for $0\leq t \leq T_0$.
Moreover we have the "crude" estimate from Corollary \ref{cor: energy estimate} and Lemmas \ref{lem: norm varphi sin} and \ref{lem: norm u_l}:
\begin{align}
\|v\|_{H^k(\R)} &\leq \|u_{\omega,\lambda}(t)\|_{H^k(\R)}+\|u^{\omega,\lambda}(t)\|_{H^k(\R)} \nonumber \\
& \leq \|u^{\omega,\lambda}(0)\|_{H^k(\R)}+\|u^{\omega,\lambda}(t)\|_{H^k(\R)} \nonumber \\
& \lesssim \lambda^{k-s} \label{eq: Hk estimate v}
\end{align}
for any $k>s$ and $0\leq t\leq T_0$. Interpolating between \eqref{eq: L2 estimate v} and \eqref{eq: Hk estimate v} for $k>s$ gives the result.
\end{proof}

We can now conclude the proof of Theorem \ref{thm: line}.
\begin{proof}[Proof of Theorem \ref{thm: line}]
Assume first that $s>\frac{3}{2}$. Let $T_0$ be as in Lemma \ref{lem: exact approximate}. Set $\omega_1=1$ and $\omega_2=-1$. At time $t=0$ we have
\begin{equation}
\label{eq: time t=0}
\|u_{1,\lambda}(0)-u_{-1,\lambda}(0)\|_{H^s(\R)}=2\lambda^{-1}\|\tilde{\varphi}_\lambda \|_{H^s(\R)}\lesssim \lambda^{-1+\delta/2}.
\end{equation}
For time $0<t\leq T_0$, we make repeated use of the triangle inequality and Lemma \ref{lem: exact approximate} to get
\begin{align}
\|u_{1,\lambda}-u_{-1,\lambda}\|_{H^s(\R)} \geq & \|u^{1,\lambda}-u^{-1,\lambda}\|_{H^s(\R)} \label{eq: time t>0}\\
& -\|u^{1,\lambda}-u_{1,\lambda}\|_{H^s(\R)}  \nonumber \\
& - \|u^{-1,\lambda}-u_{-1,\lambda}\|_{H^s(\R)} \nonumber \\
\gtrsim & \|u^{1,\lambda}-u^{-1,\lambda}\|_{H^s(\R)} -\lambda^{-\varepsilon(1-\frac{s}{k})}. \nonumber
\end{align}
By basic trigonometry,
\begin{equation}
\label{eq: trigonometry}
u^{1,\lambda}-u^{-1,\lambda}=u_{l,1,\lambda}-u_{l,-1,\lambda}+2\lambda^{-\delta/2-s}\varphi_\lambda \sin(-\lambda m(\lambda) t+ \lambda x)\sin(t). 
\end{equation}
From \eqref{eq: u_l Hs norm}, \eqref{eq: time t>0} and \eqref{eq: trigonometry} we conclude that
\begin{equation}
\label{eq: t>0}
\|u_{1,\lambda}-u_{-1,\lambda}\|_{H^s(\R)}\gtrsim \sin(t) +\lambda^{-1+\delta/2}-\lambda^{-\varepsilon(1-\frac{s}{k})}.
\end{equation}
As $1<\delta<2$ and $k>s$, the terms involving $\lambda$ on the right hand side go to $0$ as $\lambda\rightarrow \infty$. Picking an increasing sequence $\lbrace \lambda_n \rbrace_n$ such that $\lambda_n\rightarrow \infty$ as $n\rightarrow \infty$, we get two sequences of solutions $\lbrace u_{\pm 1, \lambda_n}\rbrace_n$ that at $t=0$ converges in $H^s(\R)$ as $n\rightarrow \infty$ (cf. \eqref{eq: time t=0}), while at times $t>0$ are bounded apart independently of $n$. This proves part (i).

Assume now that $m$ satisfies in addition the lower bound $|m(\xi)|\gtrsim |\xi|^r$ for $|\xi|\gg 1$ for some $1<r<\gamma$. Let $0<s< \frac{r}{2}$ and assume Theorem \ref{thm: local well-posedness} is true in this case. That is, given $u_0\in H^s(\R)$, an $H^s(\R)$ solution exists up to some time $T$ that depends only on $\|u_0\|_{H^s(\R)}$. Thus, letting $u_{\omega,\lambda}$ be as in Lemma \ref{lem: exact approximate}, there is a $T_0$ that can be considered independent of $\lambda$ such that $u_{\omega,\lambda}\in H^s(\R)$ for $0\leq t\leq T_0$ for all $\lambda>1$. The arguments establishing \eqref{eq: L2 estimate v} are valid when $\delta/2+s>1$, which can be achieved for any $s>0$ by choosing $2-2s<\delta<2$. Note that this condition on $\delta$ and the one in Lemma \ref{lem: error estimate line} can be satisfied simultaneously and thus $\delta\in (1,2)$ can be chosen such that \eqref{eq: L2 estimate v} holds also in the present case. The difficulty is to establish \eqref{eq: Hk estimate v} for $0\leq t\leq T_0$, as Lemma \ref{thm: existence time} is valid only for $s>\frac{3}{2}$ and we can therefore only use that result and Corollary \ref{cor: energy estimate} to get estimates for $0\leq t \leq T\simeq \lambda^{s-\delta}$ for any $\delta>\frac{3}{2}$ when $s\leq \frac{3}{2}$. For a solution $u$ of \eqref{eq: whitham} the quantities
\begin{equation*}
\int_\R u^2\dx
\end{equation*}
and
\begin{equation*}
\frac{1}{2}\int_\R uL u\dx-\frac{1}{6}\int_\R u^3 \dx
\end{equation*}
are preserved (see Lemma 1 in \cite{abh}). If $|m(\xi)|\gtrsim |\xi|^r$, then 
\begin{equation*}
\int_\R |uL u|+u^2\dx\gtrsim \|u\|_{H^{r/2}(\R)}^2.
\end{equation*}
For $r>\frac{1}{3}$ we can interpolate to get
\begin{equation*}
\left| \int_\R u^3\dx\right| \lesssim \|u\|_{H^{1/6}(\R)}^3\lesssim \|u\|_{L^2(\R)}^{3-1/r}\|u\|_{H^{r/2}(\R)}^{1/r}.
\end{equation*}
If $r>\frac{1}{2}$, then $\left|\int_\R uL u \dx\right|$ increases faster than $\left| \int_\R u^3\dx\right|$ as $\|u\|_{H^{r/2}(\R)}$ increases, as the $L^2(\R)$ norm is preserved. This implies that
\begin{equation}
\|u(t)\|_{H^{r/2}(\R)}\simeq \|u(0)\|_{H^{r/2}(\R)}
\end{equation}
for all $t\in \R$, and hence \eqref{eq: Hk estimate v} holds for $k=r/2$, and interpolating between \eqref{eq: L2 estimate v} and \eqref{eq: Hk estimate v} we get that Lemma \ref{lem: exact approximate} holds for $0<s< \frac{r}{2}$. Thus we get \eqref{eq: time t=0} and \eqref{eq: t>0} exactly as before.
\end{proof}

\noindent
{\bf Acknowledgements.} The author would like to thank Alexander Himonas and Vera Hur for their valuable input and advice during the research for this paper.

\medskip

\bibliographystyle{plain}
\bibliography{non-uniform}

\end{document}